\documentclass[12pt]{amsart}
\usepackage[utf8]{inputenc}
\usepackage[T1]{fontenc}
\usepackage{amsfonts}
\usepackage[leqno]{amsmath}
\usepackage{accents}
\usepackage{graphicx}
\usepackage{amssymb, comment, float}
\usepackage[dvipsnames]{xcolor}
\usepackage[foot]{amsaddr}
\usepackage[shortlabels]{enumitem}
\setlist[itemize]{leftmargin=11mm}
\setlist[enumerate]{leftmargin=11mm}
\usepackage{mathdots}
\usepackage[a4paper, footskip=0.5in,  headheight = 0.5in, top=1.25in, bottom=1.25in,  right=1in,  left=1in]{geometry}
\usepackage{hyperref}
\hypersetup{
colorlinks,
linkcolor=black,
urlcolor=Blue,
citecolor=black,}
\usepackage[capitalise,nameinlink]{cleveref}
\usepackage[center]{caption}
\usepackage{eufrak}
\usepackage{marvosym}     
\usepackage{latexsym}
\usepackage[nodayofweek]{datetime}
\usepackage{tikz}
\usepackage{subfig}
\usepackage{thm-restate}
\usepackage{verbatim}
\usepackage[none]{hyphenat} 
\usepackage{cite}
\usepackage{stmaryrd}
\newcommand{\yl}{\hspace{-0.3mm}\mathrel{\scalebox{1.25}{$\Yleft$}}\hspace{-0.3mm}}

\newtheorem{statement}{Statement}[section]
\newtheorem{theorem}[statement]{Theorem}
\newtheorem{lemma}[statement]{Lemma}
\newtheorem{conjecture}[statement]{Conjecture}
\newtheorem{corollary}[statement]{Corollary}

\DeclareMathOperator{\tw}{{\sf tw}}
\DeclareMathOperator{\pw}{{\sf pw}}
\DeclareMathOperator{\ta}{{\text{\sf{tree}}-\alpha}}
\DeclareMathOperator{\pa}{{\text{\sf{path}}-\alpha}}
  \DeclareMathOperator{\we}{{\text{\sf{w}}}}

\newcommand{\mca}{\mathcal}
\newcommand{\poi}{\mathbb{N}} 

\newcounter{tbox}
\newcommand{\sta}[1]{\medskip\medskip\refstepcounter{tbox}\noindent{\parbox{\textwidth}{(\thetbox) \emph{#1}}}\vspace*{0.3cm}}
\newcommand{\mylongtitle}[1]{%
  \ifodd\value{page}%
    \protect\parbox{0.97\linewidth}{#1}\hfill%
  \else%
    \hfill\protect\parbox{0.97\linewidth}{#1}%
  \fi%
}

\title[Tree-alpha and excluding finitely many graphs]{Tree-alpha and excluding finitely many graphs}

\author{Sepehr Hajebi$^{\dagger}$}

\author{Sophie Spirkl$^{\dagger \ast}$}
\thanks{$^{\dagger}$ Department of Combinatorics and Optimization, University of Waterloo, Waterloo, Ontario, Canada.}
\thanks{$^{\ast}$ We acknowledge the support of the Natural Sciences and Engineering Research Council of Canada (NSERC), [funding reference number RGPIN-2020-03912].
Cette recherche a \'et\'e financ\'ee par le Conseil de recherches en sciences naturelles et en g\'enie du Canada (CRSNG), [num\'ero de r\'ef\'erence RGPIN-2020-03912]. This project was funded in part by the Government of Ontario. This research was conducted while Spirkl was an Alfred P. Sloan Fellow.}

\date{\today}

\begin{document}

 \maketitle


\begin{abstract}
We prove that a hereditary graph class $\mca{G}$ defined by finitely many excluded induced subgraphs has bounded tree-$\alpha$ if and only if it is ``$(\mathrm{tw},\omega)$-bounded'' (that is, for all $t\in \mathbb N$, the class of all $K_t$-free graphs in $\mca{G}$ has bounded treewidth). Equivalently, $\mca{G}$ has bounded tree-$\alpha$ if and only if it excludes a complete bipartite graph, a forest whose components each have at most three leaves, and the line graph of such a forest.

This resolves two conjectures of Dallard, Krnc, Kwon, Milani\v{c}, Munaro, \v{S}torgel, and Wiederrecht: the above, and a weaker one that for all $a,b\in \mathbb N$, every hereditary class that excludes $K_{a,a}$ and the $b$-vertex path has bounded tree-$\alpha$. The latter was already open even for $(a,b)\in \{(2,7),(3,5)\}$, and only recently proved for $(a,b)=(2,6)$.
\end{abstract}

\section{Introduction}
We write $\poi$ for the set of positive integers, $V(G)$ and $E(G)$ for the vertex set and edge set of a graph $G$, and $\alpha(G)$ and $\omega(G)$ for the cardinalities of a largest \textit{stable set} (a set of pairwise nonadjacent vertices) and a largest \textit{clique} (a set of pairwise adjacent vertices) in $G$. Graphs in this paper are finite, with no ``loops'' or ``parallel edges.'' A \textit{class} is a set of graphs taken up to isomorphism. See \cite{diestel} for other standard graph-theoretic terminology such as tree decompositions and the treewidth and pathwidth of a graph $G$, denoted $\tw(G)$ and $\pw(G)$.

A class $\mca{G}$ has \textit{bounded treewidth} if there exists $c\in \mathbb N$ such that $\tw(G)\leq c$ for all $G\in \mca{G}$; \textit{bounded pathwidth} is defined analogously. Over the past decade, considerable work has focused on characterizing these properties in hereditary (induced-subgraph-closed) classes. One motivation is the very pretty answer, due to Robertson and Seymour \cite{GMI,GMV}, to the same question for minor-closed classes, where bounded treewidth and pathwidth are equivalent to excluding a planar graph and a forest, respectively. Among many others, we studied the induced subgraph question, largely in joint work with Chudnovsky (and often others), including \cite{tw7, tw8, tw9, tw11, tw13, tw12, tw19, tw17, tw18, tw16}, and settled the pathwidth case in \cite{tw18}. The treewidth case continues to evade us.

A necessary condition for bounded treewidth is to exclude large complete graphs, so one may refine the notion accordingly: A class $\mca{G}$ is \textit{$(\tw,\omega)$-bounded} if there is a function $f:\poi\to \poi$ such that $\tw(G)\leq f(\omega(G))$ for every $G\in \mca{G}$. For instance, chordal graphs form a $(\tw,\omega)$-bounded class since they (are exactly the graphs that) admit tree decompositions whose bags induce complete subgraphs \cite{diestel}. 

More generally, the \textit{tree independence} (or \textit{stability}) \textit{number} of a graph $G$ \cite{DMS2,yolov}, denoted $\ta(G)$, is the minimum $c\in \mathbb N$ for which $G$ admits a tree decomposition whose bags induce subgraphs with no stable set on $c+1$ vertices, and a class $\mca{G}$ has \textit{bounded $\ta$} if there exists $c\in \mathbb N$ such that $\ta(G)\leq c$ for all $G\in \mca{G}$; the class of chordal graphs has bounded $\ta$ since $\ta(G)\leq 1$ if (and only if) $G$ is chordal. By Ramsey's theorem \cite{multiramsey}, every class of bounded $\ta$ is $(\tw,\omega)$-bounded (indeed with a polynomial bounding function). Dallard, Milani\v{c}, and \v{S}torgel \cite{DMS} made a daring conjecture that for hereditary classes the converse also holds:

\begin{conjecture}[Dallard, Milani\v{c}, \v{S}torgel; 8.5 in \cite{DMS}]\label{conj:DMS}
    A hereditary class $\mca{G}$ has bounded $\ta$ if (and only if) $\mca{G}$ is $(\tw,\omega)$-bounded.
\end{conjecture}

Chudnovsky and Trotignon \cite{CT} refuted this in a strong sense: For every $g:\poi\setminus \{1\}\to \poi$, there is a hereditary $(\tw,\omega)$-bounded class where all bounding functions are pointwise larger than $g$. (For pathwidth, the picture is more favorable: one of us recently proved \cite{polypw} that every hereditary ``$(\pw,\omega)$-bounded'' class admits a polynomial ``$(\pw,\omega)$-bounding'' function, although the ``$\pa$'' version of \Cref{conj:DMS} is also false \cite{awesome}.)

\Cref{conj:DMS} may still be true for certain types of hereditary classes. A fairly general conjecture in this direction was posed by Dallard, Krnc, Kwon, Milani\v{c}, Munaro, \v{S}torgel, and Wiederrecht \cite{DKKMMSW}, asserting that \Cref{conj:DMS} holds for hereditary classes defined by finitely many excluded induced subgraphs. For graphs $G$ and $H$, we say that $G$ is \textit{$H$-free} if no induced subgraph of $G$ is isomorphic to $H$. For a class $\mca{F}$ of graphs, we say that $G$ is \textit{$\mca{F}$-free} if $G$ is $F$-free for every $F\in \mca{F}$. It follows that a class $\mca{G}$ is hereditary if and only if it is the class of all $\mca{F}$-free graphs for some class $\mca{F}$.

\begin{conjecture}[Dallard, Krnc, Kwon, Milani\v{c}, Munaro, \v{S}torgel, and Wiederrecht; 1.2 in \cite{DKKMMSW}]\label{conj:mainclass}
Let $\mca{F}$ be a finite class of graphs. Then the class of all $\mca{F}$-free graphs has bounded $\ta$ if (and only if) it is $(\tw,\omega)$-bounded.
\end{conjecture}

\begin{figure}[t!]
    \centering
    \includegraphics[scale=0.4]{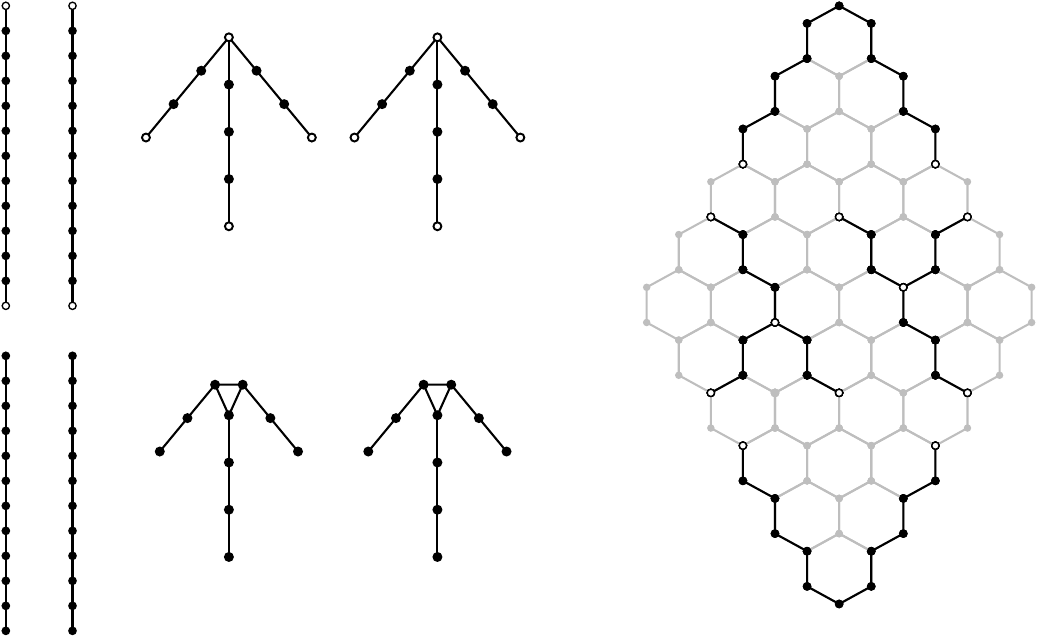}
    \caption{\raggedright Left: A subdivided multiclaw $F$ with four components (top) and its line graph (bottom). Right: The $7$-by-$7$ wall contains an induced copy of $F$, and so (the line graph of) every subdivision of the same wall contains an induced copy of (the line graph of) $F$.}
    \label{fig:wall}
\end{figure}

In this paper, we prove \Cref{conj:mainclass} -- more precisely, an equivalent formulation which is more explicit. A \textit{multiclaw} is a graph each component of which is isomorphic to an induced subgraph of $K_{1,3}$ (see \Cref{fig:wall}). It is a well-known observation \cite{lozin} that for a finite class $\mca{F}$, if the class of all $\mca{F}$-free graphs is $(\tw,\omega)$-bounded, then $\mca{F}$ contains a complete bipartite graph, a subdivided multiclaw, and the line graph of a subdivided multiclaw. (In fact, the converse is also true \cite{lozin}.) It follows that \Cref{conj:mainclass} is equivalent to:

\begin{conjecture}[Dallard, Krnc, Kwon, Milani\v{c}, Munaro, \v{S}torgel, and Wiederrecht; 1.2 in \cite{DKKMMSW}]\label{conj:mainexcluded}
Let $\mca{F}$ be a finite class of graphs. Then the class of all $\mca{F}$-free graphs has bounded $\ta$ if (and only if) there exist $F_1,F_2,F_3\in \mca{F}$ where $F_1$ is complete bipartite, $F_2$ is a subdivided multiclaw, and $F_3$ is the line graph of a subdivided multiclaw.
\end{conjecture}

We prove a stronger statement: instead of the line graph of a subdivided multiclaw, we exclude the line graphs of all subdivisions of some ``wall.'' A \textit{wall} is a hexagonal grid, and it is readily observed that for every subdivided multiclaw $F$, there is a wall $W$ such that (the line graph of) every subdivision of $W$ has an induced subgraph isomorphic to (the line graph of) $F$; see \Cref{fig:wall}.
\medskip

Our main result is the following strengthening\footnote{One may wonder whether the subdivided multiclaw can also be replaced by subdivisions of some wall. Several constructions show that this would fail. See for instance \cite{deathstar, layeredwheels}.} of \Cref{conj:mainexcluded}:

\begin{restatable}{theorem}{main}\label{thm:main}
     For all $a\in \mathbb N$, every subdivided multiclaw $F$, and every wall $W$, there is a constant $c_{\ref{thm:main}}=c_{\ref{thm:main}}(a,F,W)\in \mathbb N$ such that if $G$ is a $K_{a,a}$-free $F$-free graph with no induced subgraph isomorphic to the line graph of any subdivision of $W$, then $\ta(G)\leq c_{\ref{thm:main}}$.
\end{restatable}

There has been no shortage of partial results on Conjectures~\ref{conj:mainclass} and \ref{conj:mainexcluded}. To mention a few, it had been proved that \Cref{conj:mainexcluded} holds:

\begin{itemize}[itemsep=1mm,leftmargin=8mm]
\item with ``bounded'' weakened to ``polylogarithmic in the number of vertices'' \cite{ti5};
\item when $F_1$ is a star \cite{DKKMMSW};
\item when each component of both $F_2$ and $F_3$ has at most two vertices \cite{indmatch};
\item when $F_2=F_3$ consists of a $3$-vertex path and an isolated vertex \cite{tinypaths};
\item when $F_2=F_3$ is the $4$-vertex path \cite{DKKMMSW};
\item when $F_1=K_{2,2}$ and $F_2=F_3$ is a $5$-vertex path \cite{DKKMMSW};
\item when $F_1=K_{2,d}$ for some $d\in \mathbb N$, and $F_2=F_3$ is the $6$-vertex path \cite{k2dp6}.
\end{itemize}

Note that in the last five results, $F_2=F_3$ is (an induced subgraph of) a path. This (very) special case of \Cref{conj:mainexcluded} was singled out as a separate conjecture in \cite{DKKMMSW}:

\begin{conjecture}[Dallard, Krnc, Kwon, Milani\v{c}, Munaro, \v{S}torgel, and Wiederrecht; 1.3 in \cite{DKKMMSW}]\label{conj:path}
For all $a,b\in \mathbb N$, the class of all $K_{a,a}$-free graphs with no induced subgraph isomorphic to the $b$-vertex path has bounded $\ta$.
\end{conjecture}

Somewhat vexatiously, this remained open already for $(a,b)\in \{(2,7),(3,5)\}$, and was proved only recently \cite{k2dp6} for $(a,b)=(2,6)$. Our result settles \Cref{conj:path}.
\medskip

We conclude by pointing out that \Cref{thm:main} has strong algorithmic consequences: it implies that \textsc{Maximum Weight Independent Set}, and several other problems that are {\sf NP}-hard in general, admit polynomial-time algorithms when restricted to hereditary classes that exclude a complete bipartite graph, a subdivided multiclaw, and the line graphs of all subdivisions of a wall. See \cite{ti2,ti5} for details.
\medskip

The rest of the paper is organized as follows. In \Cref{sec:notation}, we introduce notation and terminology used throughout. In \Cref{sec:loose}, we break the proof of \Cref{thm:main} into its two main steps, Theorems~\ref{thm:mainbnddsep} and \ref{thm:ta-vs-sep}. In Sections~\ref{sec:disentangle} and \ref{sec:many}, we prove the technical results needed for the proof of \Cref{thm:mainbnddsep}, which is completed in \Cref{sec:short}. We then prove \Cref{thm:ta-vs-sep} in \Cref{sec:separators}.

\section{Preliminaries}\label{sec:notation}
Here we include the (mostly customized) notations and definitions that we will use in the rest of the paper:
\begin{itemize}[itemsep=1mm,leftmargin=8mm]

\item For integers $k,k'$, we denote by $\{k,\ldots,k'\}$ the set of all integers at least $k$ and at most $k'$. For a set $I$ of integers, we define 
$I\yl I=\{(i,j)\in I\times I: i<j\}$. Note the natural bijection between $I\yl I$ and the set of all $2$-subsets of $I$.

\item For a set $S$, we denote the power set of $S$ by $2^S$, and the set of all $k$-subsets of $S$, where $k\in \poi\cup \{0\}$, by $\binom{S}{k}$. If $S$ is a set of sets, a \textit{hitting set for $S$} is a set $X$ such that $X\cap Y\neq\varnothing$ for every $Y\in S$. In particular, every set is a hitting set for $S=\varnothing$, whereas $S=\{\varnothing\}$ admits no hitting set.

\item For graphs $G_1,\ldots, G_k$, where $k\in \mathbb N$, we write $G_1\cup \ldots\cup G_k$ for the \textit{graph union} of $G_1,\ldots, G_k$; that is, the graph whose vertex set is the union of the vertex sets of $G_1,\ldots, G_k$ and whose edge set is the union of the edge sets of $G_1,\ldots, G_k$.

\item Let $P$ be a path with $V(P)=\{x_1,\dots,x_k\}$ and $E(P)=\{x_ix_{i+1}:i\in \{1,\ldots,k-1\}\}$, where $k\in \mathbb N$. We call $x_1$ and $x_k$ the \textit{ends of $P$}. The \textit{interior of $P$}, denoted $P^*$, is the set $V(P)\setminus \{x_1,x_k\}$ of vertices. The \textit{length} of $P$ is $|E(P)|=k-1$.

\item Given a graph $H$, a \textit{subdivision} of $H$ is a graph $H'$ obtained from $H$ by replacing each edge with a path of nonzero length between the same ends such that for distinct edges of $H$, the corresponding paths are internally disjoint. For $r \in \poi\cup \{0\}$, we say that $H'$ is a \textit{$(\leq r)$-subdivision} of $H$ if each edge of $H$ is replaced by a path of length at most $r+1$, and an \textit{$r$-subdivision} if each edge is replaced by a path of length exactly $r+1$. We also say that $H'$ is a \textit{proper subdivision} of $H$ if each edge is replaced by a path of length at least two, and a \textit{pure subdivision} if either $H'=H$ or $H'$ is a proper subdivision of $H$.
\end{itemize}

Let $G$ be a graph:
\begin{itemize}[itemsep=1mm,leftmargin=8mm]
    \item A \emph{path in $G$} is an induced subgraph of $G$ that is a path. Given a set $\mca{P}$ of paths in $G$, a \textit{hitting set for $\mca{P}$} is a hitting set for $\{V(P):P\in \mca{P}\}$. For $X,Y\subseteq V(G)$, an \textit{$(X,Y)$-path in $G$} is a path $P$ in $G$ such that either $P$ has length zero and $V(P)\subseteq X\cap Y$, or $P$ has positive length, one end lies in $X\setminus Y$, the other in $Y\setminus X$, and $V(P^*)\cap (X\cup Y)=\varnothing$. We denote the set of all $(X,Y)$-paths in $G$ by $\mca{P}_G(X,Y)$, and the set of all $(X,Y)$-paths of length at most $r$ in $G$, for $r\geq 0$, by $\mca{P}^r_G(X,Y)$. When $X=\{x\}$, we write $\mca{P}_G(x,Y)$ and $\mca{P}^r_G(x,Y)$ instead of $\mca{P}_G(\{x\},Y)$ and $\mca{P}^r_G(\{x\},Y)$ (and similarly if $|Y|=1$).

\item For $X\subseteq V(G)$, we denote by $G[X]$ the subgraph of $G$ induced by $X$, and by $G\setminus X$ the subgraph of $G$ induced by $V(G)\setminus X$. We denote by $N_G(X)$ the set of all vertices in $V(G)\setminus X$ with at least one neighbor in $X$. If $X=\{x\}$, we write $G\setminus x$ instead of $G\setminus \{x\}$, and $N_G(x)$ instead of $N_G(\{x\})$. We also define $N_G[x]=N_G(x)\cup \{x\}$. 

\item For $X,Y\subseteq V(G)$, we say that \textit{$X$ and $Y$ are anticomplete in $G$} if $X\cap Y=\varnothing$ and there is no edge in $G$ with an end in $X$ and an end in $Y$; in this case, we also say that \textit{$G[X]$ and $G[Y]$ are anticomplete in $G$}. If $X=\{x\}$, we say that \textit{$x$ is anticomplete to $Y$ in $G$} to mean that $\{x\}$ and $Y$ are anticomplete in $G$. We say that \textit{$X$ is complete to $Y$ in $G$} if $X\cap Y=\varnothing$ and $xy\in E(G)$ for all $(x,y)\in X\times Y$. If $X=\{x\}$, we say that \textit{$x$ is complete to $Y$ in $G$} to mean that $\{x\}$ is complete to $Y$ in $G$.

\item For $b\in \mathbb N$, a \textit{$b$-system} in $G$ is a set of non-empty subsets of $V(G)$, each of cardinality at most $b$ (thus, $S=\varnothing$ is allowed, whereas $S=\{\varnothing\}$ is not). We say that a $b$-system $\mca{S}$ in $G$ is \textit{anticomplete} if every two distinct sets $S_1,S_2\in \mca{S}$ are anticomplete in $G$.
\end{itemize}

\section{Loose $\alpha$-separability}\label{sec:loose}

A graph $G$ is \textit{$d$-separable}, for $d\in \poi$, if for all $x,y\in V(G)$, there is a hitting set $X\subseteq V(G)\setminus \{x,y\}$ for $\mca{P}_G(x,y)$ with $|X|\leq d$, unless $x=y$ or $xy\in E(G)$ (in which case no such hitting set exists). This notion, introduced in \cite{milanicsep}, has proved useful for bounding treewidth in hereditary classes (see, for instance, \cite{polypw, tw17, tw18}). There is also an ``$\alpha$-version'' of separability, used extensively (though sometimes implicitly) in the literature on $\ta$ (see, for instance, \cite{ti2, ti5, ti6, alphabrb, k2dp6, logDMS}): We say that a graph $G$ is \textit{$d$-$\alpha$-separable} if for all distinct nonadjacent $x,y\in V(G)$, there is a hitting set $X\subseteq V(G)\setminus \{x,y\}$ for $\mca{P}_G(x,y)$ with $\alpha(G[X])\leq d$. 

The key to our proof of \Cref{thm:main}, however, is a ``loose'' variant of $\alpha$-separability: For $b,c,d\in \mathbb N$, a graph $G$ is \textit{$(b,c)$-loosely $d$-$\alpha$-separable} if for every $b$-system $\mca{S}$ in $G$, either
\begin{itemize}[itemsep=1mm, leftmargin=8mm]
    \item there is a hitting set $X\subseteq V(G)$ for $\mca{S}$ with $\alpha(G[X])\leq c$; or
    \item there exist $S_1,S_2\in \mca{S}$ such that $\mca{P}_G(S_1,S_2)$ has a hitting set $Y\subseteq V(G)\setminus (S_1\cup S_2)$ with $\alpha(G[Y])\leq d$; in particular, $S_1$ and $S_2$ are anticomplete in $G$.
\end{itemize}
Roughly, the larger $b,c\in \mathbb N$, the ``coarser'' this view is relative to the usual $d$-$\alpha$-separability in $G$. Specifically, $G$ is $d$-$\alpha$-separable if and only if it is $(1,1)$-loosely $d$-$\alpha$-separable.
\medskip

The two main steps in the proof of \Cref{thm:main} are the following:
\begin{theorem}\label{thm:mainbnddsep}
    For all $a,b\in \mathbb N$ and every subdivided multiclaw $F$, there are constants $c_{\ref{thm:mainbnddsep}}=c_{\ref{thm:mainbnddsep}}(a,b,F)\in \mathbb N$ and $d_{\ref{thm:mainbnddsep}}=d_{\ref{thm:mainbnddsep}}(a,b,F)\in \mathbb N$ such that every $K_{a,a}$-free $F$-free graph is $(b,c_{\ref{thm:mainbnddsep}})$-loosely $d_{\ref{thm:mainbnddsep}}$-$\alpha$-separable.
\end{theorem}

\begin{restatable}{theorem}{tavssep}\label{thm:ta-vs-sep}
For every subdivided multiclaw $F$ and every wall $W$, there is a constant $b_{\ref{thm:ta-vs-sep}}=b_{\ref{thm:ta-vs-sep}}(F,W)\in \mathbb N$ with the following property. Let $G$ be an $F$-free graph with no induced subgraph isomorphic to the line graph of any subdivision of $W$, and let $c,d\in \mathbb N$ such that $G$ is $(b_{\ref{thm:ta-vs-sep}}, c)$-loosely $d$-$\alpha$-separable. Then
$\ta(G)\leq 10c+5d$.
\end{restatable}

\Cref{thm:main} is immediate from the above result:

\main*

\begin{proof}
    Let 
    $b=b_{\ref{thm:ta-vs-sep}}(F,W)$, let $c=c_{\ref{thm:mainbnddsep}}(a,b,F)$ and let $d=d_{\ref{thm:mainbnddsep}}(a,b,F)$. Let
$$c_{\ref{thm:main}}=c_{\ref{thm:main}}(a,F,W)=10c+5d.$$
    From \Cref{thm:mainbnddsep} and the choice of $c,d$, it follows that $G$ is $(b,c)$-loosely $d$-$\alpha$-separable. Thus, from \Cref{thm:ta-vs-sep} and the choice of $b$, it follows that $\ta(G)\leq 10c+5d=c_{\ref{thm:main}}$. 
\end{proof}
It remains to prove Theorems~\ref{thm:mainbnddsep} and \ref{thm:ta-vs-sep}, which we will do in Sections~\ref{sec:short} and \ref{sec:separators}.
\section{Forcing identical ends and disentangling anticomplete $b$-systems}\label{sec:disentangle}

The purpose of this section is to prove two technical lemmas that will be used in the proof of \Cref{thm:mainbnddsep}. The first lemma uses Ramsey's theorem for uniform hypergraphs (which we will also use in the next section):

\begin{theorem}[Ramsey \cite{multiramsey}]\label{thm:multiramsey}
For all $r,s,t\in \mathbb N$, there is a constant  $c_{\ref{thm:multiramsey}}=c_{\ref{thm:multiramsey}}(r,s,t)\in \mathbb N$ with the following property. Let $I$ be a non-empty set with $|I|\leq r$, let $A$ be a set with $|A|\geq c_{\ref{thm:multiramsey}}$ and let  $\Phi:\binom{A}{s}\rightarrow I$ be a function. Then there exist $i\in I$ and a $t$-subset $B$ of $A$ such that $\Phi(S)=i$ for all $S\in \binom{B}{s}$.
\end{theorem}

\begin{lemma}\label{lem:endsramsey}
    For all $b,s\in \mathbb N$, there is a constant $c_{\ref{lem:endsramsey}}=c_{\ref{lem:endsramsey}}(b,s)\in \mathbb N$ with the following property. Let $\{S_1,\ldots, S_{c_{\ref{lem:endsramsey}}}\}$ be a $b$-system in a graph $G$ and let $\mca{P}_{i,j}\subseteq \mca{P}_G(S_{i},S_{j})$ for each $(i,j)\in  \{1,\ldots, c_{\ref{lem:endsramsey}}\}\yl \{1,\ldots, c_{\ref{lem:endsramsey}}\}$, where $|\mca{P}_{i,j}|\geq b^2r$ for some $r\in \mathbb N$. Then there exist
    \begin{itemize}[leftmargin=8mm, itemsep=1mm]
        \item $I\subseteq \{1,\ldots, c_{\ref{lem:endsramsey}}\}$ with $|I|=s$; \item $\mca{R}_{i,j}\subseteq \mca{P}_{i,j}$ with $|\mca{R}_{i,j}|=r$ for each $(i,j)\in I\yl I$; and 
        \item vertices $x_i, y_i\in S_i$ for each $i\in I$ (where $x_i=y_i$ is possible);
    \end{itemize}
    such that for every $(i,j)\in \{1,\ldots, c_{\ref{lem:endsramsey}}\}\yl \{1,\ldots, c_{\ref{lem:endsramsey}}\}$, we have $\mca{R}_{i,j}\subseteq \mca{P}_G(x_{i},y_{j})$.
\end{lemma}
\begin{proof}
Let
    $$c_{\ref{lem:endsramsey}}=c_{\ref{lem:endsramsey}}(b,s)=c_{\ref{thm:multiramsey}}(b^2,2,s).$$
    
For each $i\in \{1,\ldots, c_{\ref{lem:endsramsey}}\}$, choose an injection $f_i:S_i\to \{1,\ldots, b\}$ (which exists because $1\leq |S_i|\leq b$). For every $(i,j)\in \{1,\ldots, c_{\ref{lem:endsramsey}}\}\yl \{1,\ldots, c_{\ref{lem:endsramsey}}\}$, since $|\mca{P}_{i,j}|\geq b^2r$, it follows that there exist 
$(k_{i,j},k_{j,i})\in \{1,\ldots, b\}\times \{1,\ldots, b\}$
and $\mca{R}_{i,j}\subseteq \mca{P}_{i,j}$ with $|\mca{R}_{i,j}|=r$ such that every path in $\mca{R}_{i,j}$ has $f^{-1}_{i}(k_{i,j})$ and $f^{-1}_{j}(k_{j,i})$ as its ends. 

Let $\Phi:\binom{\{1,\ldots, c_{\ref{lem:endsramsey}}\}}{2}\to \{1,\ldots, b\}\times \{1,\ldots, b\}$
be the function that maps each $2$-subset $\{i,j\}$ of $\{1,\ldots, c_{\ref{lem:endsramsey}}\}$ with $i<j$ to $\Phi(\{i,j\})=(k_{i,j},k_{j,i})$. Since $c_{\ref{lem:endsramsey}}=c_{\ref{thm:multiramsey}}(b^2,2,s)$, it follows from \Cref{thm:multiramsey} applied to $\Phi$ that there exist $k_1,k_2\in \{1,\ldots, b\}$ and $I\subseteq \{1,\ldots, c_{\ref{lem:endsramsey}}\}$ with $|I|=s$ such that for every $(i,j)\in I\yl I$, we have $\Phi(\{i,j\})=(k_{i,j},k_{j,i})=(k_1,k_2)$. Now, for each $i\in I$, let $x_i=f^{-1}_i(k_1)$ and let $y_i=f^{-1}_i(k_2)$; thus, $x_{i},y_i\in S_i$. Also, for all $(i,j)\in I\yl I$, every path in $\mca{R}_{i,j}$ has $f^{-1}_{i}(k_{i,j})=f^{-1}_{i}(k_1)=x_{i}$ and $f^{-1}_{j}(k_{j,i})=f^{-1}_{j}(k_2)=y_j$ as its ends. This completes the proof of \Cref{lem:endsramsey}.
\end{proof}

For our second lemma, along with \Cref{thm:multiramsey}, we need the following special case of the Graham-Rothschild theorem \cite{GRgeneral} (which will be used in the next section as well):

\begin{theorem}[Graham, Rothschild, Spencer \cite{productramsey}]\label{thm:productramsey}
For all  $r,s,t\in \mathbb N$, there is a constant  $c_{\ref{thm:productramsey}}=c_{\ref{thm:productramsey}}(r,s,t)\in \mathbb N$ with the following property. Let $I$ be a non-empty set with $|I|\leq r$, let $A_1,\ldots, A_s$ be sets of cardinality at least $c_{\ref{thm:productramsey}}$, and let $\Phi:A_1\times \cdots \times A_s\rightarrow I$ be a function. Then there exist $i\in I$ as well as a $t$-subset $B_j$ of $A_j$ for each $j\in \{1,\ldots, s\}$, such that $\Phi(b)=i$ for all $b\in B_1\times \cdots\times B_s$.
\end{theorem}

\begin{lemma}\label{lem:ramseysmallsystems}
    For all $a,b,r,s,t\in \mathbb N$, there are constants $c_{\ref{lem:ramseysmallsystems}}=c_{\ref{lem:ramseysmallsystems}}(a,b,r,s,t)\in \mathbb N$ and $d_{\ref{lem:ramseysmallsystems}}=d_{\ref{lem:ramseysmallsystems}}(a,b,r,s,t)\in \mathbb N$ with the following property. Let $\{S_1,\ldots, S_{c_{\ref{lem:ramseysmallsystems}}}\}$ be an anticomplete $b$-system in a $K_{a,a}$-free graph $G$. For each $(i,j)\in  \{1,\ldots, c_{\ref{lem:ramseysmallsystems}}\}\yl \{1,\ldots, c_{\ref{lem:ramseysmallsystems}}\}$, let $\mca{R}_{i,j}$ be an anticomplete $r$-system in $G$ with $|\mca{R}_{i,j}|\geq d_{\ref{lem:ramseysmallsystems}}$. Then there exist
    
    \begin{itemize}[leftmargin=8mm, itemsep=1mm]
        \item $I\subseteq \{1,\ldots, c_{\ref{lem:ramseysmallsystems}}\}$ with $|I|=s$; and
        \item $\mca{T}_{i,j}\subseteq \mca{R}_{i,j}$ with $|\mca{T}_{i,j}|=t$ for each $(i,j)\in I\yl I$;
    \end{itemize}
    such that the following hold:
    
    \begin{enumerate}[{\rm (a)}, leftmargin=8mm, itemsep=1mm]

\item\label{lem:ramseysmallsystems_a} For all $(i,j),(i',j')\in I\yl I$, each $T\in \mca{T}_{i,j}$ and every $T'\in \mca{T}_{i',j'}$, the sets $T,T'$ are anticomplete in $G$.
        \item\label{lem:ramseysmallsystems_b} For all distinct $i,i',i''\in I$ with $i'<i''$ and every $T\in \mca{T}_{i',i''}$, the sets $S_i,T$ are anticomplete in $G$.
    \end{enumerate}
\end{lemma}
\begin{proof}
    Let
$$c_{\ref{lem:ramseysmallsystems}}=c_{\ref{lem:ramseysmallsystems}}(a,b,r,s,t)=c_{\ref{thm:multiramsey}}\left(2^{6brt},3,\max\{3a,s,6\}\right).$$
Let $c=\binom{c_{\ref{lem:ramseysmallsystems}}}{2}$, and let
$$d_{\ref{lem:ramseysmallsystems}}=d_{\ref{lem:ramseysmallsystems}}(a,b,r,s,t)=c_{\ref{thm:productramsey}}\left(2^{2c^2r^2}, c, \max\{a,t,2\}\right).$$

   We begin with the following:

    \sta{\label{st:antibetweenpairs} There exists $\mca{T}_{i,j}\subseteq \mca{R}_{i,j}$ with $|\mca{T}_{i,j}|\geq t$ for each $(i,j)\in \{1,\ldots, c_{\ref{lem:ramseysmallsystems}}\}\yl \{1,\ldots, c_{\ref{lem:ramseysmallsystems}}\}$, such that for all $(i,j),(i',j')\in \{1,\ldots, c_{\ref{lem:ramseysmallsystems}}\}\yl \{1,\ldots, c_{\ref{lem:ramseysmallsystems}}\}$, each $T\in \mca{T}_{i,j}$ and every $T'\in \mca{T}_{i',j'}$, the sets $T,T'$ are anticomplete in $G$.}

    Let $\{i_1,j_1\},\ldots, \{i_c,j_c\}$ be an enumeration of $2$-subsets of $\{1,\ldots, c_{\ref{lem:ramseysmallsystems}}\}$, where $i_k<j_k$ for every $k\in \{1,\ldots, c\}$. For each $R\in \bigcup_{k=1}^c\mca{R}_{i_k,j_k}$, choose an injection $f_R:R\to \{1,\ldots, r\}$ (which exists because $1\leq |R|\leq r$). Let 
    $$\mca{I}=\{1,\ldots, c\}\times \{1,\ldots, c\}\times \{1,\ldots, r\}\times \{1,\ldots, r\}.$$
    Let
$\phi_1,\phi_2:\mca{R}_{i_1,j_1}\times \cdots \times \mca{R}_{i_c,j_c}\to 2^{\mca{I}}$ be functions with the rules
$$\phi_1(R_1,\ldots, R_c)= \left\{(k,l,m,n):\ k<l,\ f_{R_k}^{-1}(m)= f_{R_l}^{-1}(n)\right\}\subseteq \mca{I};$$
$$\phi_2(R_1,\ldots, R_c)= \left\{(k,l,m,n):\ k<l,\ f_{R_k}^{-1}(m)f_{R_l}^{-1}(n)\in E(G)\right\}\subseteq \mca{I};$$
and let  
$\Phi:\mca{R}_{i_1,j_1}\times \cdots \times \mca{R}_{i_c,j_c}\to 2^{\mca{I}}\times 2^{\mca{I}}$ be the function with the rule
$$\Phi(R_1,\ldots, R_c) =\left(\phi_1(R_1,\ldots, R_c),\phi_2(R_1,\ldots, R_c)\right).$$
 Since $|2^{\mca{I}}\times 2^{\mca{I}}|=2^{2c^2r^2}$ and $|\mca{R}_{i_1,j_1}|= \cdots = |\mca{R}_{i_c,j_c}|=d_{\ref{lem:ramseysmallsystems}}=c_{\ref{thm:productramsey}}(2^{2c^2r^2},c,\max\{a,t,2\})$, by \Cref{thm:productramsey} applied to $\Phi$, there exist $\mca{I}_1,\mca{I}_2\subseteq \mca{I}$, as well as $\mca{T}_{i_{\ell},j_{\ell}}\subseteq \mca{R}_{i_{\ell},j_{\ell}}$ with $|\mca{T}_{i_{\ell},j_{\ell}}|=\max\{a,t,2\}$ for each $\ell\in \{1,\ldots, c\}$, such that 
 $$\Phi(T_1,\ldots, T_c)=\left(\phi_1(T_1,\ldots, T_c),\phi_2(T_1,\ldots, T_c)\right)=(\mca{I}_1,\mca{I}_2)$$
 for every $(T_1,\ldots, T_c)\in \mca{T}_{i_1,j_1}\times \cdots \times \mca{T}_{i_c,j_c}$. Assume $(k,l,m,n)\in \mca{I}_1$. Since $|\mca{T}_{i_k,j_k}|,|\mca{T}_{i_l,j_l}|\geq 2$, we may choose $T\in \mca{T}_{i_k,j_k}$ and distinct $T',T''\in \mca{T}_{i_l,j_l}$. Since $(k,l,m,n)\in \mca{I}_1$, it follows that $f_{T}^{-1}(m)= f_{T'}^{-1}(n)$ and $f_{T}^{-1}(m)= f_{T''}^{-1}(n)$. But then we have $f_{T'}^{-1}(n)= f_{T''}^{-1}(n)\in T'\cap T''$, contradicting the assumption that $\mca{T}_{i_l,j_l}\subseteq \mca{R}_{i_l,j_l}$ is an anticomplete $r$-system in $G$. Assume that $(k,l,m,n)\in \mca{I}_2$. Since $|\mca{T}_{i_k,j_k}|=|\mca{T}_{i_l,j_l}|\geq a$, we may choose $\mca{A}\subseteq \mca{T}_{i_k,j_k}$ and $\mca{A}'\subseteq \mca{T}_{i_l,j_l}$ with $|\mca{A}|=|\mca{A}'|=a$. Let $A=\{f_{T}^{-1}(m):T\in \mca{A}\}$ and let $A'=\{f_{T}^{-1}(n): T\in \mca{A}'\}$. Since $\mca{T}_{i_k,j_k}\subseteq \mca{R}_{i_k,j_k}$ and $\mca{T}_{i_l,j_l}\subseteq \mca{R}_{i_l,j_l}$ are anticomplete $r$-systems in $G$, it follows that $A,A'$ are stable sets in $G$ with $|A|=|A'|=a$. Since $(k,l,m,n)\in I_2$, it follows that $A$ is complete to $A'$ in $G$ (in particular, $A\cap A'=\varnothing$). But then $G[A\cup A']$ is isomorphic to $K_{a,a}$, again a contradiction. We deduce that $\mca{I}_1=\mca{I}_2=\varnothing$; and so for all $k,l\in \{1,\ldots, c\}$, each $T\in \mca{T}_{i_k,j_k}$ and every $T'\in \mca{T}_{i_l,j_l}$, the sets $T,T'$ are anticomplete in $G$. Moreover, recall that $|\mca{T}_{i_1,j_1}|= \cdots = |\mca{T}_{i_c,j_c}|=\max\{a,t,2\}\geq t$. This proves \eqref{st:antibetweenpairs}.
 \medskip

Henceforth, let $(\mca{T}_{i,j}: 1\leq i<j\leq c_{\ref{lem:ramseysmallsystems}})$ be as given by \eqref{st:antibetweenpairs}. For each $i\in \{1,\ldots, c_{\ref{lem:ramseysmallsystems}}\}$, choose an injection $g_i:S_i\to \{1,\ldots, b\}$ (which exists because $1\leq |S_i|\leq b$). For each $(i,j)\in \{1,\ldots, c_{\ref{lem:ramseysmallsystems}}\}\yl \{1,\ldots, c_{\ref{lem:ramseysmallsystems}}\}$, let
 $U_{i,j}=\bigcup_{T\in \mca{T}_{i,j}}T$. Then $t\leq |U_{i,j}|\leq rt$, and so we may choose an injection 
 $h_{i,j}:U_{i,j} \to \{1,\ldots, rt\}$.
 
Let 
$$\mca{J}=\{1,\ldots, b\}\times \{1,\ldots, rt\}.$$
For each $j\in \{1,2,3\}$, let 
$$\phi_{1,j},\phi_{2,j}:\binom{\{1,\ldots, c_{\ref{lem:ramseysmallsystems}}\}}{3}\to 2^{\mca{J}}$$
be functions with the following rules: 

Write $\{j',j''\}=\{1,2,3\}\setminus\{j\}$ where $j'<j''$. For every $Z=\{i_1,i_2,i_3\}\subseteq \{1,\ldots, c_{\ref{lem:ramseysmallsystems}}\}$  with $i_1<i_2<i_3$, we have
$$\phi_{1,j}(Z)=\left\{(k,l):\  g_{i_j}^{-1}(k)= h_{i_{j'},i_{j''}}^{-1}(l)\right\}\subseteq \mca{J}$$
and
$$\phi_{2,j}(Z) = \left\{(k,l):\  g_{i_j}^{-1}(k)h_{i_{j'},i_{j''}}^{-1}(l)\in E(G)\right\} \subseteq \mca{J}.$$
Let 
$$\Phi:\binom{\{1,\ldots, c_{\ref{lem:ramseysmallsystems}}\}}{3}\to \mca{I}=\underbrace{2^{\mca{J}}\times \cdots \times 2^{\mca{J}}}_{6\text{ times}}$$
be the function with the rule: 
$$\Phi(Z)=\left(\phi_{1,j}(Z), \phi_{2,j}(Z): j=1,2,3\right).$$

 Since $|\mca{I}|=2^{6brt}$ and $c_{\ref{lem:ramseysmallsystems}}=c_{\ref{thm:multiramsey}}(2^{6brt},3,\max\{3a,s,6\})$, by \Cref{thm:multiramsey} applied to $\Phi$, there exist $J_{1,1}, J_{1,2}, J_{1,3}, J_{2,1}, J_{2,2}, J_{2,3}\subseteq  \mca{J}$ and $I\subseteq \{1,\ldots, c_{\ref{lem:ramseysmallsystems}}\}$ with $|I|=\max\{3a,s,6\}$ such that for every $3$-subset $Z$ of $I$ and all $j\in \{1,2,3\}$, we have $\phi_{1,j}(Z)=J_{1,j}$ and $\phi_{2,j}(Z)=J_{2,j}$. We further claim that:
 
 \sta{\label{st:J1empty} $J_{1,1}=J_{1,2}= J_{1,3}=\varnothing$.}

 Suppose not. By symmetry, we may assume that $J_{1,1}\neq\varnothing$; say, $(k,l)\in J_{1,1}$. Since $|I|\geq 6$, we may choose $i_1,i'_1,i_2,i'_2,i_3,i'_3\in I$ with $i_1<i'_1<i_2<i'_2<i_3<i'_3$. Since 
 $$(k,l)\in J_{1,1}=\phi_{1,1} \left(\left\{i_1,i_{2},i_{3}\right\}\right)=\phi_{1,1}\left(\left\{i'_1,i_{2},i_{3}\right\}\right),$$
 it follows that
 $g_{i_1}^{-1}(k)= h_{i_{2},i_{3}}^{-1}(l)$ and $g_{i'_1}^{-1}(k)= h_{i_{2},i_{3}}^{-1}(l)$. But then we have
 $$g_{i_1}^{-1}(k)= g_{i'_1}^{-1}(k)\in S_{i_1}\cap S_{i'_1};$$
 which contradicts the assumption that $S_1,\ldots, S_{c_{\ref{lem:ramseysmallsystems}}}$ are pairwise anticomplete in $G$. This proves \eqref{st:J1empty}.

  \sta{\label{st:J2empty} $J_{2,1}=J_{2,2}= J_{2,3}=\varnothing$.}

 Suppose not. By symmetry, we may assume that $J_{2,1}\neq\varnothing$; say, $(k,l)\in J_{2,1}$. Since $|I|\geq 3a$, there exist $I_1,I_2,I_3\subseteq I$ with $|I_1|=|I_2|=|I_3|=a$ such that $\max I_1<\min I_2$ and  $\max I_2<\min I_3$. Choose a bijection $\psi: I_{2}\to I_{3}$. Let 
 $A=\{g_{i}^{-1}(k): i\in I_1\}$ and let $A'=\{h_{i,\psi(i)}^{-1}(l):i\in I_{2}\}$. Since $S_1,\ldots, S_{c_{\ref{lem:ramseysmallsystems}}}$ are pairwise anticomplete in $G$, it follows that $A$ is a stable set in $G$ with $|A|=a$. By \eqref{st:antibetweenpairs}, the sets $(U_{i,\psi(i)}:i\in I_{2})$ are pairwise anticomplete in $G$, which in turn implies that $A'$ is a stable set in $G$ with $|A'|=a$. Moreover, $A$ is complete to $A'$ in $G$ because $(k,l)\in J_{2,1}=\phi_{2,1}(\{i_1,i_2,i_3\})$ for all $(i_1,i_2,i_3)\in I_1\times I_2\times I_3$. But then $G[A\cup A']$ is isomorphic to $K_{a,a}$, a contradiction. This proves \eqref{st:J2empty}.
 \medskip

 Now, by \eqref{st:antibetweenpairs}, $(\mca{T}_{i,j}: (i,j)\in I\yl I)$ satisfy \ref{lem:ramseysmallsystems}\ref{lem:ramseysmallsystems_a}, and from \eqref{st:J1empty} and \eqref{st:J2empty},  it follows that for all distinct $i,i',i''\in I$ with $i'<i''$, the sets $S_i,U_{i',i''}$ are anticomplete in $G$; thus, \ref{lem:ramseysmallsystems}\ref{lem:ramseysmallsystems_b} hold as well. Recall also that $|I|=\max\{3a,s,6\}\geq s$, and $|\mca{T}_{i,j}|=t$ for every $(i,j)\in I\yl I$. This completes the proof of \Cref{lem:ramseysmallsystems}.
\end{proof}

\section{Several $K_{a,a}$-free graphs with a dense union}\label{sec:many}

In this section, we prove the following, which is at the heart of the proof of \Cref{thm:main}:

\begin{theorem}\label{thm:manygraphs}
For all $a,b,q\in \mathbb N$, there is a constant $c_{\ref{thm:manygraphs}}=c_{\ref{thm:manygraphs}}(a,b,q)\in \mathbb N$ with the following property. Let $G_1,\ldots, G_b$ be $K_{a,a}$-free graphs with $V(G_1)=\cdots=V(G_b)=V$ and $\alpha(G_1\cup\cdots\cup G_b)<q$. Then there exist $V_1,\ldots, V_b\subseteq V$ with $V_1\cup \cdots \cup V_b=V$ such that 
for every $i\in \{1,\ldots, b\}$, we have $\alpha(G_i[V_i])\leq c_{\ref{thm:manygraphs}}$.
\end{theorem}

The case $b=2$ may be proved with an explicit bound:

\begin{lemma}\label{lem:twographs}
Let $a,q\in \mathbb N$. Let $G_1,G_2$ be $K_{a,a}$-free graphs with $V(G_1)=V(G_2)=V$ such that $\alpha(G_1\cup G_2)<q$. Then there exist $V_1,V_2\subseteq V$ with $V_1\cup V_2=V$ such that for every $i\in \{1,2\}$, we have
$\alpha(G_i[V_i])<\left(2a\right)^{a(q-1)}$.
\end{lemma}

\begin{proof}

The proof is by induction on $q$ (for fixed $a$). If $q=1$, then $V=\varnothing$ and we may choose one of $V_1=V_2=V$. Assume that $q\geq 2$. For $\beta\in \mathbb N$, we say that $X\subseteq V$ is \textit{$\beta$-covered} if there exist $X_1,X_2\subseteq X$ with $X_1\cup X_2=X$ such that $\alpha(G_1[X_1])<  \beta$ and $\alpha(G_2[X_2])< \beta$. Our goal is therefore to show that $V$ is $(2a)^{a(q-1)}$-covered.

Define the sequence $(\beta_n: n\in \mathbb N)$ recursively as follows: Let $\beta_1=\left(2a\right)^{a(q-2)}$, and for every $n\geq 2$, let
$$\beta_n=a(\beta_{n-1}+\beta_1).$$

For each $v\in V$, let 
$M_1(v)=N_{G_2}(v)\setminus N_{G_1}(v)$ and let $ M_2(v)=N_{G_1}(v)\setminus N_{G_2}(v)$.

\sta{\label{st:bigsets} Let $X_0,X_1,X_2$ be subsets of $V$ that are not necessarily distinct. Let $n\in \mathbb N$ such that for every $v\in X_0$, at least one of $X_1\cap M_1(v)$ or $X_2\cap M_2(v)$ is $\beta_n$-covered. Then at least one of $X_0,X_1$ or $X_2$ is $\beta_{n+1}$-covered.}

Suppose that $X_0$ is not $\beta_{n+1}$-covered. For each $i\in \{1,2\}$, let $Y_i$ be the set of all vertices $v\in X_{0}$ for which $X_i\cap M_i(v)$ is $\beta_n$-covered. Then $Y_{1}\cup Y_2=X_0$. Since $X_0$ is not $\beta_{n+1}$-covered, it follows that for some $i\in \{1,2\}$, we have $\alpha(G_{i}[Y_i])\geq \beta_{n+1}>a$. Choose a stable set $A$ in $G_i$ with $|A|=a$ such that $A\subseteq Y_i$. Let 
$$B=\bigcup_{v\in A} (X_i\cap M_{i}(v)),\quad  
C=\bigcup_{v\in A} \left(X_i\setminus N_{G_1\cup G_2}(v)\right),\quad D=X_i\cap\left(\bigcap_{v\in A} N_{G_i}(v)\right).$$
Observe that $X_i=B\cup C\cup D$. Moreover,

\begin{itemize}[itemsep=1mm,leftmargin=8mm]
\item Since $X_i\cap M_i(v)$ is $\beta_n$-covered for every $v\in A\subseteq Y_i$, and since $|A|=a$, it follows that $B$ is $(a(\beta_n-1)+1)$-covered.
\item Since every vertex $v\in A$ is anticomplete to $X_i\setminus N_{G_1\cup G_2}[v]$ in $G_1\cup G_2$, it follows that $\alpha((G_1\cup G_2)[X_i\setminus N_{G_1\cup G_2}[v]])<\alpha(G_1\cup G_2)\leq q-1$; thus, by the inductive hypothesis, $X_i\setminus N_{G_1\cup G_2}[v]$ is $\beta_1$-covered, and so $X_i\setminus N_{G_1\cup G_2}(v)$ is $(\beta_1+1)$-covered. Since $|A|=a$, we deduce that $C$ is $(a\beta_1+1)$-covered.

\item Since $G_i$ is $K_{a,a}$-free and $A$ is complete to $D$ in $G_i$, it follows that $\alpha(G_i[D])<a$; in particular, $D$ is $a$-covered.
\end{itemize}
But now $X_i=B\cup C\cup D$ is $(a(\beta_n-1)+a\beta_1+a)$-covered, as required. This proves \eqref{st:bigsets}.
\medskip

Let $\ell,m\in \mathbb N$ with $\ell\leq m$. An \textit{$m$-chain of length $\ell$} is an $\ell$-tuple $(v_1,\ldots, v_{\ell})$ of vertices in $V$ such that:

\begin{itemize}[itemsep=1mm,leftmargin=8mm]
\item $v_j\in M_1(v_i)$ for all $(i,j)\in \{1,\ldots, \ell\}\yl\{1,\ldots, \ell\}$; and
\item $M_1(v_1)\cap \cdots \cap M_1(v_{\ell})$ and $M_2(v_1)\cap \cdots \cap M_2(v_{\ell})$ are not $\beta_{m-\ell+1}$-covered. 
\end{itemize}

\sta{\label{st:nolongchain} There is no $a$-chain of length $a$.}

Suppose not. Let $(v_1,\ldots, v_a)$ be an $a$-chain of length $a$, let $A=\{v_1,\ldots, v_a\}$ and let $Z=M_2(v_1)\cap \cdots \cap M_2(v_{a})$. By the first bullet in the definition, $A$ is a stable set in $G_1$ with $|A|=a$, and by the second bullet, $Z$ is not $\beta_1$-covered. Since $M_2(v_i)\subseteq N_{G_1}(v_i)$ for all $i\in \{1,\ldots, a\}$, it follows that $A$ is complete to $Z$ in $G_1$. Since $G_1$ is $K_{a,a}$-free, it follows that $\alpha(G_1[Z])<a\leq \beta_1$. But now $Z$ is $\beta_1$-covered, a contradiction. This proves \eqref{st:nolongchain}.
\medskip

\sta{\label{st:longchain} For every $v\in V$, at least one of $M_1(v)$ or $M_2(v)$ is $\beta_a$-covered.}

Suppose not. Choose $v\in V$ such that $M_1(v)$ and $M_2(v)$ are not $\beta_a$-covered. Then $v$ is an $a$-chain of length $1$. Let $I$ be the set of all $\ell\in \{1,\ldots, a\}$ for which there is no $a$-chain of length $\ell$. Then $1\notin I$. Also, by \eqref{st:nolongchain}, we have $I\neq \varnothing$. Let $\ell_0=\min I$; then $\ell_0\geq 2$. By the minimality of $\ell_0$, there is an $a$-chain $(v_1,\ldots, v_{\ell_0-1})$ of length $\ell_0-1$. Let $Z_1=M_1(v_1)\cap \cdots \cap M_1(v_{\ell_0-1})$ and let $Z_2=M_2(v_1)\cap \cdots \cap M_2(v_{\ell_0-1})$. By the second bullet in the definition, $Z_1$ and $Z_2$ are not $\beta_{a-\ell_0+2}$-covered. Therefore, by \eqref{st:bigsets} (applied to $X_0=X_1=Z_1$ and $X_2=Z_2$), there is a vertex $v_{\ell_0}\in Z_1$ for which $Z_1\cap M_1(v_{\ell_0})$ and $Z_2\cap M_2(v_{\ell_0})$ are not $\beta_{a-\ell_0+1}$-covered. But now $(v_1,\ldots, v_{\ell_0-1},v_{\ell_0})$ is an $a$-chain of length $\ell_0$, contradicting the choice of $\ell_0\in I$. This proves \eqref{st:longchain}.

\sta{\label{st:zetabound} $\beta_n\leq (2a)^{n-1}\beta_1$ for all $n\in \mathbb N$.}

This is clear for $n=1$. Inductively, if $\beta_{n-1}\leq (2a)^{n-2}\beta_1$ for some $n\geq 2$, then
$$\beta_{n}=a(\beta_{n-1}+\beta_1)\leq  a\left((2a)^{n-2}\beta_1+\beta_1\right)\leq a\left(2(2a)^{n-2}\beta_1\right)=(2a)^{n-1}\beta_1.$$
This proves \eqref{st:zetabound}.
\medskip
  
Now, by \eqref{st:longchain}, for every $v\in V$, at least one of $M_1(v)$ or $M_2(v)$ is $\beta_a$-covered. Consequently, it follows from \eqref{st:bigsets} (applied to $X_0=X_1=X_2=V$) that $V$ is $\beta_{a+1}$-covered. Also, by \eqref{st:zetabound},
\[\beta_{a+1}\leq (2a)^a\beta_1=(2a)^a\left(2a\right)^{a(q-2)}=\left(2a\right)^{a(q-1)}.\]
This completes the proof of \Cref{lem:twographs}.
\end{proof}

We combine \Cref{lem:twographs} with Theorems~\ref{thm:multiramsey} and \ref{thm:productramsey} to prove \Cref{thm:manygraphs}:

\begin{proof}[Proof of \Cref{thm:manygraphs}] Throughout, let $a\in \mathbb N$ be fixed. We begin by constructing a function $f:\mathbb N\times \mathbb N\to \mathbb N$. First,  for every $m\in \mathbb N$, let 
$$\gamma_m=c_{\ref{thm:productramsey}}\left(m,2,a\right).$$
Observe that $\gamma_m\geq a$ for all $m\in \mathbb N$. Define $f$ recursively: For every $n\in \mathbb N$, let $f(1,n)=n$.  For every $(m,n)\in \mathbb N\times \mathbb N$ with $m\geq 2$ and $f(m-1,n')$ defined for all $n'\in \mathbb N$, let
    $$f(m,n)=f\left(m-1,\left(2\gamma_{m-1}\right)^{\gamma_{m-1}(n-1)}\right).$$
This concludes the definition of $f$. We deduce that:

\sta{\label{st:increaseonn} Let $m\in \mathbb N$. Then $f(m,n)\geq n$ for all $n\in \mathbb N$.}

This is clear for $m=1$. Inductively, if $m\geq 2$ such that $f(m-1,n')\geq n'$ for all $n'\in \mathbb N$, then we have
$$f(m,n)=f\left(m-1,\left(2\gamma_{m-1}\right)^{\gamma_{m-1}(n-1)}\right)\geq \left(2\gamma_{m-1}\right)^{\gamma_{m-1}(n-1)}\geq 2^{n-1}\geq n$$
for all $n\in \mathbb N$. This proves \eqref{st:increaseonn}.
\medskip

Now, with $f$ defined as above, we will prove by induction on $b\in \mathbb N$ that for every $q\in \mathbb N$,
$$c_{\ref{thm:manygraphs}}=c_{\ref{thm:manygraphs}}(a,b,q)=f(b,q)$$
satisfies the theorem. If $b=1$, then $\alpha(G_1)<q=f(1,q)=c_{\ref{thm:manygraphs}}$, and we may choose $V_1=V$. From here on, assume that $b\geq 2$. Let 
$G'_2=G_2\cup\cdots \cup G_b$. We claim that:

\sta{\label{st:G2free} $G_1$ and $G'_2$ are $K_{\gamma_{b-1},\gamma_{b-1}}$-free.}

This is immediate for $G_1$ because $\gamma_{b-1}\geq a$. Suppose for a contradiction that there exist $\gamma_{b-1}$-subsets $A_1,A_2$ of $V$ that are stable and complete to each other in $G'_2$. In particular, $A_1,A_2$ are stable in $G_i$ for every $i\in \{2,\ldots, b\}$, and for each $(x_1,x_2)\in A_1\times A_2$, there exists $i_{(x_1,x_2)}\in \{2,\ldots, b\}$ such that $x_1,x_2$ are adjacent in $G_{i_{(x_1,x_2)}}$. Let $\Phi:A_1\times A_2\to \{2,\ldots, b\}$ be the function with the rule 
$$\Phi(x_1,x_2)=i_{(x_1,x_2)}.$$
Since $|A_1|=|A_2|=\gamma_{b-1}=c_{\ref{thm:productramsey}}(b-1, 2,a)$, by \Cref{thm:productramsey} applied to $\Phi$, there exist $i_0\in \{2,\ldots, b\}$, $B_1\subseteq A_1$, and $B_2\subseteq A_2$, such that $|B_1|=|B_2|=a$ and for every $(x_1,x_2)\in B_1\times  B_2$, we have $\Phi((x_1,x_2))=i_{(x_1,x_2)}=i_0$. But now
$G_{i_0}[B_1\cup B_2]$ is isomorphic to $K_{a,a}$, a contradiction. This proves \eqref{st:G2free}.
\medskip

By \eqref{st:G2free} and since $\alpha(G_1\cup G'_2)=\alpha(G_1\cup G_2\cup \cdots\cup G_b)<q$, we can apply \Cref{lem:twographs} to $G_1$ and $G'_2$, and obtain $V_1,V'_2\subseteq V$ with $V_1\cup V'_2=V$ such that 
$$\alpha(G_1[V_1]), \alpha(G'_2[V'_2])< \left(2\gamma_{b-1}\right)^{\gamma_{b-1}(q-1)}.$$
Since $b\geq 2$, from \eqref{st:increaseonn}, it follows that
$$\alpha(G_1[V_1])< \left(2\gamma_{b-1}\right)^{\gamma_{b-1}(q-1)}\leq f\left(b-1,\left(2\gamma_{b-1}\right)^{\gamma_{b-1}(q-1)}\right)=f(b,q)=c_{\ref{thm:manygraphs}}.$$
Also, since
$\alpha\left(G_2[V'_2]\cup \cdots\cup G_b[V'_2]\right)=\alpha(G'_2[V'_2])<\left(2\gamma_{b-1}\right)^{\gamma_{b-1}(q-1)}$, by the inductive hypothesis, there exist $V_2,\ldots, V_b\subseteq V'_2$ with $V_2\cup \cdots \cup V_b=V'_2$ such that for every $i\in \{2,\ldots, b\}$, we have $$\alpha(G_i[V_i])\leq f(b-1,\left(2\gamma_{b-1}\right)^{\gamma_{b-1}(q-1)})\leq f(b,q)=c_{\ref{thm:manygraphs}}.$$
Hence, we obtain $V_1,V_2,\ldots, V_b\subseteq V$ with $V_1\cup V_2\cup \cdots \cup V_b=V_1\cup V'_2=V$ such that $\alpha(G_i[V_i])\leq c_{\ref{thm:manygraphs}}$ for every $i\in \{1,\ldots, b\}$. This completes the proof of \Cref{thm:manygraphs}.
\end{proof}

\section{Short paths across a $b$-system}\label{sec:short}

Here, we prove \Cref{thm:mainbnddsep}. In fact, we prove the following stronger result. A \textit{multistar} is a graph with each component isomorphic to $K_{1,d}$ for some $d\geq 0$. In particular, every multiclaw is a multistar; thus, the following implies \Cref{thm:mainbnddsep}:

\begin{restatable}{theorem}{bettermainbnddsep}\label{thm:bettermainbnddsep}
    For all $a,b\in \mathbb N$ and every subdivided multistar $F$, there are constants $c_{\ref{thm:bettermainbnddsep}}=c_{\ref{thm:bettermainbnddsep}}(a,b,F)\in \mathbb N$ and $d_{\ref{thm:bettermainbnddsep}}=d_{\ref{thm:bettermainbnddsep}}(a,b,F)\in \mathbb N$ such that every $K_{a,a}$-free $F$-free graph is $(b,c_{\ref{thm:bettermainbnddsep}})$-loosely $d_{\ref{thm:bettermainbnddsep}}$-$\alpha$-separable.
\end{restatable}

We need two lemmas:

\begin{lemma}\label{lem:hit-vs-anti}
    For all $a,b,q\in \mathbb N$, there is a constant $c_{\ref{lem:hit-vs-anti}}=c_{\ref{lem:hit-vs-anti}}(a,b,q)\in \mathbb N$ such that for every $K_{a,a}$-free graph $G$ and every $b$-system $\mca{S}$ in $G$, either there is a hitting set $X\subseteq V(G)$ for $\mca{S}$ with $\alpha(G[X])< c_{\ref{lem:hit-vs-anti}}$, or there is an anticomplete $b$-system $\mca{Q}\subseteq\mca{S}$ in $G$ with $|\mca{Q}|\geq q$.
\end{lemma}
\begin{proof}
   Throughout, let $a,b,q\in \mathbb N$ be fixed. For every $m\in \mathbb N$, let $$\zeta_{m}=c_{\ref{thm:multiramsey}}\left(2^{2m^2},2,\max\{2a, q, 3\}\right),\quad \eta_m=c_{\ref{thm:manygraphs}}(a+1,m,\zeta_m).$$
Let
    $$c_{\ref{lem:hit-vs-anti}}=c_{\ref{lem:hit-vs-anti}}(a,b,q)=\sum_{m=1}^b m\eta_{m}.$$

Let $G$ be a $K_{a,a}$-free graph and let $\mca{S}$ be a $b$-system in $G$ that admits no hitting set $X\subseteq V(G)$ with $\alpha(G[X])\leq \sum_{m=1}^b m\eta_{m}$. We prove that there is an anticomplete $b$-system $\mca{Q}\subseteq\mca{S}$ in $G$ with $|\mca{Q}|\geq q$.
    
    For each $m\in \{1,\ldots, b\}$, let $\mca{S}_m=\{S\in \mca{S}: |S|=m\}$. Since $\mca{S}=\mca{S}_1\cup\cdots\cup \mca{S}_b$, it follows that for some $m\in \{1,\ldots, b\}$, there is no hitting set $X\subseteq V(G)$ for $\mca{S}_m$ with $\alpha(G[X])< m\eta_m$; in particular, $\mca{S}_m\neq \varnothing$ because  $X=\varnothing$ is not a hitting set for $\mca{S}_m$.

    Let $|\mca{S}_m|=n\geq 1$; say, $\mca{S}_m=\{S_1,\ldots,S_{n}\}$. For each $j\in \{1,\ldots, n\}$, let 
    $$S_j=\{x_{i,j}:i\in \{1,\ldots, m\}\}.$$
    
Let $V=\{v_1,\ldots, v_n\}$ be a set. For each $i\in \{1,\ldots, m\}$, let $G_i$ be the graph with $V(G_i)=V$ such that for all distinct $j,j'\in \{1,\ldots, n\}$, we have $v_{j}v_{j'}\in E(G_i)$ if and only if either $x_{i,j}=x_{i,j'}$ or $x_{i,j}x_{i,j'}\in E(G)$. We claim that:

\sta{\label{st:blowup} $G_1,\ldots, G_m$ are all $K_{a+1,a+1}$-free.}

Suppose for a contradiction that for some $i\in \{1,\ldots, m\}$, there exist pairwise distinct $j_1,\ldots, j_{a+1},j'_1,\ldots, j'_{a+1}\in \{1,\ldots, n\}$
such that $\{v_{j_1},\ldots, v_{j_{a+1}}\}$ and $\{v_{j'_1},\ldots, v_{j'_{a+1}}\}$ are stable and complete to each other in $G_i$. Since $G$ is $K_{a,a}$-free, it follows that there exist $k,l\in \{1,\ldots, a+1\}$ for which $x_{i,j_k}=x_{i,j'_l}$. Choose $k'\in \{1,\ldots, a+1\}$ with $k\neq k'$ (this is possible because $a+1\geq 2$). Now, on one hand, $x_{i,j_k}$ and $x_{i,j_{k'}}$ are distinct and nonadjacent in $G$ since $v_{j_k}v_{j_{k'}}\notin E(G_i)$, and on the other hand, $x_{i,j_k}$ and $x_{i,j_{k'}}$ are the same or adjacent in $G$ since $x_{i,j_k}=x_{i,j'_l}$ and $v_{j_{k'}}v_{j'_{l}}\in E(G_i)$, a contradiction. This proves \eqref{st:blowup}.

\sta{\label{st:commonstableset} There exists $J\subseteq \{1,\ldots, n\}$ with $|J|=\zeta_m$ such that for every $i\in \{1,\ldots, m\}$, the vertices $(x_{i,j}:j\in J)$ are pairwise distinct and nonadjacent in $G$.}

Suppose for a contradiction that no such $J$ exists. Then $\alpha(G_1\cup\cdots\cup G_m)< \zeta_m$. By \eqref{st:blowup} and since $\eta_m=c_{\ref{thm:manygraphs}}(a+1,m,\zeta_m)$, we can apply \Cref{thm:manygraphs} to $G_1,\ldots, G_m$ and write $J_1\cup \cdots \cup J_m=\{1,\ldots, n\}$ such that $\alpha(G_i[\{v_j:j\in J_i\}])<\eta_m$ for every $i\in \{1,\ldots, m\}$.
Note that $\alpha(G[\{x_{i,j}:j\in J_i\}])=\alpha(G_i[\{v_j:j\in J_i\}])$ for every $i\in \{1,\ldots, m\}$. Moreover, for every $j\in \{1,\ldots,n\}$, there exists $i\in \{1,\ldots, m\}$ such that $j\in J_i$, which in turn implies that $x_{i,j}\in S_j\cap \{x_{i,j}:j\in J_i\}$. But now $X=\bigcup_{i=1}^m \{x_{i,j}:j\in J_i\}$ is a hitting set for $\mca{S}_m$ with $\alpha(G[X])<m\eta_m$, a contradiction. This proves \eqref{st:commonstableset}.
\medskip

Henceforth, let $J$ be as given by \eqref{st:commonstableset}. Let 
$$I=\{1,\ldots, m\}\times \{1,\ldots, m\}.$$
Let $\phi_1,\phi_2:  J\yl J\to 2^{I}$ be functions with the rules 
$$\phi_1(j,j')= \{(i,i')\in I: x_{i,j}=x_{i',j'}\}\subseteq I;$$
$$\phi_2(j,j')=\{(i,i')\in I: x_{i,j}x_{i',j'}\in E(G)\}\subseteq I.$$ 
Let $\Phi:\binom{J}{2}\to 2^{I}\times 2^I$ be the function that maps each $2$-subset $\{j,j'\}$ of $J$ with $j<j'$ to
$$\Phi(\{j,j'\})= \left(\phi_1(j,j'),\phi_2(j,j')\right).$$

Since $|2^I\times 2^I|=2^{2m^2}$ and $|J|=\zeta_m=c_{\ref{thm:multiramsey}}(2^{2m^2},2,\max\{2a, q, 3\})$, by \Cref{thm:multiramsey} applied to $\Phi$, there exist $I_1,I_2\subseteq I$ and $K\subseteq J$ with $|K|=\max\{2a, q, 3\}$ such that for every $(j,j')\in K\yl K$, we have $\phi_1(j,j')=I_1$ and $\phi_2(j,j')=I_2$. We further claim that:

\sta{\label{st:I1empty}$I_1=\varnothing$.}

Suppose not; say, $(i,i')\in I_1$. Since $|K|\geq 3$, there exist $j,j',j''\in K$ with $j<j'<j''$. Since $(i,i')\in I_1=\phi_1(j,j'')=\phi_1(j',j'')$, it follows that $x_{i,j}=x_{i',j''}$ and $x_{i,j'}=x_{i',j''}$. But now $x_{i,j}=x_{i,j'}$, a contradiction to \eqref{st:commonstableset}. This proves \eqref{st:I1empty}.

\sta{\label{st:I2empty}$I_2=\varnothing$.}

Suppose not; say, $(i,i')\in I_2$. Since $|K|\geq 2a$, there exist $A,A'\subseteq K$ with $|A|=|A'|=a$ such that $\max A<\min A'$. By \eqref{st:commonstableset}, $Y=\{x_{i,j}:j\in A\}$ and $Y'=\{x_{i,j'}:j'\in A'\}$ are stable sets in $G$ with $|Y|=|Y'|=a$. Moreover, for each $j\in A$ and every $j'\in A'$, since $(i,i')\in I_2=\phi_2(j,j')$, it follows that $x_{i,j}x_{i',j'}\in E(G)$ (and in particular $x_{i,j}\neq x_{i',j'}$). But now $G[Y\cup Y']$ is isomorphic to $K_{a,a}$, a contradiction. This proves \eqref{st:I2empty}.
\medskip

From \eqref{st:I1empty} and \eqref{st:I2empty}, it follows that $\mca{Q}=\{S_j:j\in K\}$ is an anticomplete $b$-system in $G$ where $\mca{Q}\subseteq \mca{S}_m\subseteq \mca{S}$ and $|\mca{Q}|=|K|\geq q$. This completes the proof of \Cref{lem:hit-vs-anti}. 
\end{proof}

Next, we prove a stronger version of \Cref{lem:hit-vs-anti} (note that \Cref{lem:hit-vs-anti} is equivalent to \Cref{lem:hit-vs-far} in the case $r=1$). The proof involves applications of the technical lemmas~\ref{lem:endsramsey} and \ref{lem:ramseysmallsystems} from \Cref{sec:disentangle}:

\begin{lemma}\label{lem:hit-vs-far}
    For all $a,b,q,r\in \mathbb N$, there are constants $c_{\ref{lem:hit-vs-far}}=c_{\ref{lem:hit-vs-far}}(a,b,q,r)\in \mathbb N$ and $d_{\ref{lem:hit-vs-far}}=d_{\ref{lem:hit-vs-far}}(a,b,q,r)\in \mathbb N$ such that for every graph $G$ with no induced subgraph isomorphic to any pure $(\leq r-1)$-subdivision of $K_{a,a}$, and every $b$-system $\mca{S}$ in $G$, either
    \begin{enumerate}[{\rm (a)}, leftmargin=8mm, itemsep=1mm]
        \item\label{lem:hit-vs-far_a}  there is a hitting set $X\subseteq V(G)$ for $\mca{S}$ with $\alpha(G[X])< c_{\ref{lem:hit-vs-far}}$; or
       \item\label{lem:hit-vs-far_b} there is a $b$-system $\mca{Q}\subseteq\mca{S}$ in $G$ with $|\mca{Q}|\geq q$ such that for all distinct $S_1,S_2\in \mca{Q}$, $\mca{P}^r_G(S_1,S_2)$ has a hitting set $Y\subseteq V(G)\setminus (S_1\cup S_2)$ with $\alpha(G[Y])< d_{\ref{lem:hit-vs-far}}$.
        \end{enumerate}
\end{lemma}

\begin{proof}
Let
$$\theta=c_{\ref{lem:ramseysmallsystems}}(a,2,r,2a,1);\quad \kappa=c_{\ref{lem:endsramsey}}(b,\theta);\quad \lambda=\max\{q,\kappa\};\quad \mu=c_{\ref{thm:multiramsey}}(2,2,\lambda).$$
Let 
$$\nu=d_{\ref{lem:ramseysmallsystems}}(a,2,r,2a,1).$$
We will prove that 
$$c_{\ref{lem:hit-vs-far}}=c_{\ref{lem:hit-vs-far}}(a,b,q,r)=c_{\ref{lem:hit-vs-anti}}(a,b,\mu)\quad \text{ and }\quad d_{\ref{lem:hit-vs-far}}=d_{\ref{lem:hit-vs-far}}(a,b,q,r)=c_{\ref{lem:hit-vs-anti}}(a,r,b^2\nu)$$
satisfy the theorem.

Let $G$ be a graph with no induced subgraph isomorphic to any pure $(\leq r-1)$-subdivision of $K_{a,a}$, and let $\mca{S}$ be a $b$-system in $G$. Suppose for a contradiction that none of \ref{lem:hit-vs-far}\ref{lem:hit-vs-far_a} and \ref{lem:hit-vs-far}\ref{lem:hit-vs-far_b} holds. By \Cref{lem:hit-vs-anti} applied to $\mca{S}$, since there is no hitting set $X\subseteq V(G)$ for $\mca{S}$ with $\alpha(G[X])< c_{\ref{lem:hit-vs-far}}=c_{\ref{lem:hit-vs-anti}}(a,b,\mu)$, it follows that there is an anticomplete $b$-system $\mca{M}\subseteq \mca{S}$ with $|\mca{M}|=\mu$. Since $\mu=c_{\ref{thm:multiramsey}}(2,2,\lambda)$, it follows from \Cref{thm:multiramsey} that there exist $\lambda$ sets $S_1,\ldots, S_{\lambda}\in \mca{M}$ such that either
\begin{itemize}[leftmargin=8mm, itemsep=1mm]
    \item for all distinct $i,j\in \{1,\ldots, \lambda\}$, $\mca{P}_G^r(S_i,S_j)$ admits a hitting set $Y\subseteq V(G)\setminus (S_i\cup S_j)$ with $\alpha(G[Y])< d_{\ref{lem:hit-vs-far}}$; or
    \item for all distinct $i,j\in \{1,\ldots, \lambda\}$, $\mca{P}_G^r(S_i,S_j)$ admits no hitting set $Y\subseteq V(G)\setminus (S_i\cup S_j)$ with $\alpha(G[Y])< d_{\ref{lem:hit-vs-far}}$.
\end{itemize}
In the former case, \ref{lem:hit-vs-far}\ref{lem:hit-vs-far_b} holds because $\lambda\geq q$, a contradiction. So we may assume that the latter case occurs. By the choice of $d_{\ref{lem:hit-vs-far}}=c_{\ref{lem:hit-vs-anti}}(a,r,b^2\nu)$, it follows that for every $(i,j)\in \{1,\ldots, \lambda\}\yl \{1,\ldots, \lambda\}$, there exists $\mca{P}_{i,j}\subseteq \mca{P}_G^r(S_i,S_j)$ with $|\mca{P}_{i,j}|=b^2\nu$ such that the interiors of the paths in $\mca{P}_{i,j}$ are pairwise anticomplete in $G$. Moreover, since $\lambda\geq \kappa=c_{\ref{lem:endsramsey}}(b,\theta)$, it follows from \Cref{lem:endsramsey} that there exist $I\subseteq \{1,\ldots, \lambda\}$ with $|I|=\theta$, as well as $\mca{R}_{i,j}\subseteq \mca{P}_{i,j}$ with $|\mca{R}_{i,j}|=\nu$ for each $(i,j)\in I\yl I$, and vertices $x_i,y_i\in S_i$ for each $i\in I$,  such that $\mca{R}_{i,j}\subseteq \mca{P}_G(x_{i},y_{j})$ for all $(i,j)\in I\yl I$.

Now, $\{\{x_i,y_i\}:i\in I\}$ is an anticomplete $2$-system in $G$ of cardinality $\theta$. Moreover, $\{P^*:P\in \mca{R}_{i,j}\}$ is an anticomplete $r$-system in $G$ of cardinality $\nu$ for every $(i,j)\in I\yl I$. By \Cref{lem:ramseysmallsystems} and the choice of $\theta,\nu$, there exist $J\subseteq I$ with $|J|=2a$, as well as $P_{i,j}\in \mca{R}_{i,j}$ for each $(i,j)\in J\yl J$, such that for all distinct $(i,j),(i',j')\in J\yl J$, the sets $P_{i,j},P_{i',j'}$ are anticomplete in $G$, and for all distinct $i,i',i''\in I$ with $i'<i''$, the sets $\{x_i,y_i\}$ and $P^*_{i',i''}$ are anticomplete in $G$. Since $|J|=2a$, we may choose $J_1,J_2\subseteq J$ with $|J_1|=|J_2|=a$ such that $\max J_1<\min J_2$. But now the graph 
$$\bigcup_{(i,j)\in J_1\times J_2}P_{i,j}$$
is an induced subgraph of $G$ isomorphic to a proper $(\leq r-1)$-subdivision of $K_{a,a}$, a contradiction. This completes the proof of \Cref{lem:hit-vs-far}.
\end{proof}

We are now ready to prove \Cref{thm:bettermainbnddsep}, which we restate:

\bettermainbnddsep*

\begin{proof} Let $r=|V(F)|$. Let
$$\xi=c_{\ref{lem:ramseysmallsystems}}(a,2,r,r+1,r);\quad \xi'=d_{\ref{lem:ramseysmallsystems}}(a,2,r,r+1,r).$$
Let
$$\rho=c_{\ref{lem:endsramsey}}(b,\xi);\quad \sigma=d_{\ref{lem:hit-vs-far}}(a,b,\rho,r);\quad \tau=c_{\ref{lem:hit-vs-anti}}(a,r,b^2\xi').$$
We will prove that 
$$c_{\ref{thm:bettermainbnddsep}}=c_{\ref{thm:bettermainbnddsep}}(a,b,F)=c_{\ref{lem:hit-vs-far}}(a,b,\rho,r)\quad \text{ and }\quad d_{\ref{thm:bettermainbnddsep}}=d_{\ref{thm:bettermainbnddsep}}(a,b,F)=\sigma+\tau$$
satisfy the theorem. 

Suppose for a contradiction that for some $K_{a,a}$-free $F$-free graph $G$, there exists a $b$-system $\mca{S}$ in $G$ such that
\begin{itemize}
    \item $\mca{S}$ has no hitting set $X\subseteq V(G)$ with $\alpha(G[X])\leq  c_{\ref{thm:bettermainbnddsep}}=c_{\ref{lem:hit-vs-far}}(a,b,\rho,r)$; and
    \item for all distinct $S_1,S_2\in \mca{S}$, $\mca{P}_G(S_1,S_2)$ has no hitting set $Y\subseteq V(G)\setminus (S_1\cup S_2)$ with $\alpha(G[Y])\leq  d_{\ref{thm:bettermainbnddsep}}=\sigma+\tau$.
\end{itemize}

By the first bullet, it follows from \Cref{lem:hit-vs-far} applied to $\mca{S}$ that:

\sta{\label{st:ahappens} There exist $S_1,\ldots, S_{\rho}\in\mca{S}$ such that for all distinct $i,j\in \{1,\ldots, \rho\}$, $\mca{P}^{r}_G(S_i,S_j)$ has a hitting set $Y_{i,j}\subseteq V(G)\setminus (S_i\cup S_j)$ with $\alpha(G[Y_{i,j}])< d_{\ref{lem:hit-vs-far}}(a,b,\rho,r)=\sigma$.}

Let $S_1,\ldots, S_{\rho}\in\mca{S}$ be as given by \eqref{st:ahappens}. For each $(i,j)\in \{1,\ldots, \rho\}\yl \{1,\ldots, \rho\}$ and every path $P\in \mca{P}_G(S_i,S_j)\setminus \mca{P}^{r}_G(S_i,S_j)$, let $u_P$ be the end of $P$ in $S_i$, let $v_P$ be the end of $P$ in  $S_j$, and let $Q_P$ be the path of length $r$ in $P$ with $u_P\in V(Q_P)$. It follows that $u_P$ is an end of $P$, and $v_P\notin V(Q_P)$. Let $\mca{Q}_{i,j}=\{Q_P:P\in \mca{P}_G(S_i,S_j)\setminus \mca{P}^{r}_G(S_i,S_j)\}$.

For each $(i,j)\in \{1,\ldots, \rho\}\yl \{1,\ldots, \rho\}$, it follows from \eqref{st:ahappens} and the second bullet above that there is no hitting set $Z\subseteq V(G)\setminus (S_i\cup S_j)$ for $\mca{Q}_{i,j}$ with $\alpha(G[Z])< \tau$; thus, by \Cref{lem:hit-vs-anti} and the choice of $\tau=c_{\ref{lem:hit-vs-anti}}(a,r,b^2\xi')$, there exists $\mca{P}_{i,j}\subseteq \mca{P}_G(S_i,S_j)\setminus \mca{P}^{r}_G(S_i,S_j)$ with $|\mca{P}_{i,j}|\geq b^2\xi'$ for which $(V(Q_P)\setminus \{u_P\}:P\in \mca{P}_{i,j})$ are pairwise anticomplete in $G$.

On the other hand, since $\rho=c_{\ref{lem:endsramsey}}(b,\xi)$, it follows from \Cref{lem:endsramsey} that there exist $I\subseteq \{1,\ldots, \rho\}$ with $|I|=\xi$, as well as $\mca{R}_{i,j}\subseteq \mca{P}_{i,j}$ with $|\mca{R}_{i,j}|=\xi'$ for each $(i,j)\in I\yl I$, and vertices $x_i,y_i\in S_i$ for each $i\in I$,  such that $\mca{R}_{i,j}\subseteq \mca{P}_G(x_{i},y_{j})$ for all $(i,j)\in I\yl I$; in particular, $u_{P}=x_i$ for every $P\in \mca{R}_{i,j}$.

Now, $\{\{x_i,y_i\}:i\in I\}$ is an anticomplete $2$-system in $G$ of cardinality $\xi$. Moreover, $\{V(Q_P)\setminus \{x_i\}:P\in \mca{R}_{i,j}\}$ is an anticomplete $r$-system in $G$ of cardinality $\xi'$ for every $(i,j)\in I\yl I$. By \Cref{lem:ramseysmallsystems} and the choice of $\xi,\xi'$, it follows that there exist $J\subseteq I$ with $|J|=r+1$, as well as $\mca{T}_{i,j}\subseteq \mca{R}_{i,j}$ with $|\mca{T}_{i,j}|=r$ for each $(i,j)\in J\yl J$, such that for all distinct $(i,j),(i',j')\in J\yl J$, each $P\in \mca{T}_{i,j}$ and every $P'\in \mca{T}_{i',j'}$, the sets $V(Q_P)\setminus \{x_i\}, V(Q_{P'})\setminus \{x_{i'}\}$ are anticomplete in $G$, and for all distinct $i,i',i''\in I$ with $i'<i''$ and every $P\in \mca{T}_{i',i''}$, the sets $\{x_i,y_i\}$ and $V(Q_P)\setminus \{x_{i'}\}$ are anticomplete in $G$. But now for $j=\max J$, the graph 
$$H=\bigcup_{i\in J\setminus \{j\}}\bigcup_{P\in \mca{T}_{i,j}}Q_P$$
is an induced subgraph of $G$ with exactly $r$ components, each of which is isomorphic to the $(r-1)$-subdivision of $K_{1,r}$; in turn, $H$ has an induced subgraph isomorphic to $F$, a contradiction. This completes the proof of \Cref{thm:bettermainbnddsep}.
\end{proof}

\section{Separators and loose separability}\label{sec:separators}

In this last section, we prove \Cref{thm:ta-vs-sep}. We begin with a few definitions and a result from the literature.

Let $G$ be a graph. By a \textit{normal function on $G$} we mean a function $\we:V(G)\to [0,1]$ with $\sum_{v\in V(G)}\we(v)=1$. Let $\we$ be a normal function on $G$. For every $X\subseteq V(G)$, we write $\we(X)=\sum_{v\in X}\we(v)$. A \textit{$\we$-balanced separator in $G$} is a set $Z\subseteq V(G)$ of vertices in $G$ such that $\we(V(C))\leq 1/2$ for every component $C$ of $G\setminus Z$. 

In earlier joint work \cite{ti2} with Chudnovsky and Lokshtanov, we showed that:

\begin{lemma}[Chudnovsky, Hajebi, Lokshtanov, Spirkl; Lemma 7.1 in \cite{TAII}]\label{lem:sep-vs-ta}
    Let $G$ be a graph and let $d\in \mathbb N$ such that for every normal function $\we:V(G)\to [0,1]$, there is a $\we$-balanced separator $Z$ in $G$ with $\alpha(G[Z])\leq d$. Then $\ta(G)\leq 5d$.
\end{lemma}

For $b\in \mathbb N$, a graph $G$ is \textit{$b$-separator-dominated}  if for every normal function $\we$ on $G$, there exists $X\subseteq V(G)$ such that $N_G[X]$ is a $\we$-balanced separator in $G$. The main tool in the proof of \Cref{thm:ta-vs-sep} is the following:

\begin{lemma}\label{lem:ta-vs-sep-domin}
Let $b,c,d\in \mathbb N$ and let $G$ be a $(b,c)$-loosely $d$-$\alpha$-separable graph such that every induced subgraph of $G$ is $b$-separator-dominated. Then $\ta(G)\leq 10c+5d$. 
\end{lemma}
\begin{proof}
   Suppose not. Then by \Cref{lem:sep-vs-ta}, for some normal function $\we:V(G)\to [0,1]$, there is no $\we$-balanced separator $Z$ in $G$ with $\alpha(G[Z])\leq 2c+d$. 
   
   Let us say that a subset $S$ of $V(G)$ is \textit{dominant} if there exists $X\subseteq V(G)$ with $\alpha(G[X])\leq c$ such that $N_G[S]\cup X$ is a $\we$-balanced separator in $G$. Let $\mca{S}$ be the set of all dominant sets $S\subseteq V(G)$ with $|S|\in \{1,\ldots, b\}$. Then $\mca{S}$ is a $b$-system in $G$. We further claim that:

    \sta{\label{st:smallalphahitdone} $\mca{S}$ admits no hitting set $X_0\subseteq V(G)$ with $\alpha(G[X_0])\leq c$.}

  Suppose not. Let $X_0\subseteq V(G)$ be a hitting set for $\mca{S}$ with $\alpha(G[X_0])\leq c$, and let $G_0=G\setminus X_0$. Since $\alpha(G[X_0])\leq c<2c+d$, it follows that $X_0$ is not a $\we$-balanced separator in $G$, and so there is a component $C_0$ of $G_0=G\setminus X_0$ with $\we(V(C_0))>1/2$. Define the normal function $\we_0:V(G_0)\to [0,1]$ with the rule $\we_0(v)=\we(v)/\we(V(G_0))$. Since $G_0$ (as an induced subgraph of $G$) is $b$-separator-dominated, there exists $S_0\subseteq V(G_0)$ with $|S_0|\leq b$ such that $N_{G_0}[S_0]$ is a $\we_0$-balanced separator in $G_0$; that is, $\we_0(V(C))\leq 1/2$ for every component $C$ of $G_0\setminus N_{G_0}[S_0]$. In fact, since $C_0$ is a component of $G_0$ with $\we_0(V(C_0))=\we(V(C_0))/\we(V(G_0))\geq \we(V(C_0))>1/2$, it follows that $S_0\neq \varnothing$. Moreover, since $G\setminus (N_{G}[S_0]\cup X_0)=G_0\setminus N_{G_0}[S_0]$, it follows that for every component $C$ of $G\setminus (N_{G}[S_0]\cup X_0)$, we have $\we(V(C))=\we_0(V(C))\we(V(G_0))\leq \we_0(V(C))\leq 1/2$, and so $N_{G}[S_0]\cup X_0$ is a $\we$-balanced separator in $G$. Since $\alpha(G[X_0])\leq c$, it follows that $S_0$ is dominant, and since $1\leq |S_0|\leq b$, it follows that $S_0\in \mca{S}$. But now $S_0\cap X_0\neq \varnothing$, a contradiction. This proves \eqref{st:smallalphahitdone}.
   \medskip
   
By \eqref{st:smallalphahitdone} and since $G$ is $(b,c)$-loosely $d$-$\alpha$-separable, there exist $S_1,S_2\in \mca{S}$ for which $\mca{P}_G(S_1,S_2)$ has a hitting set $Y\subseteq V(G)\setminus (S_1\cup S_2)$ with $\alpha(G[Y])\leq d$; in particular, $S_1,S_2$ are anticomplete in $G$. For each $i\in \{1,2\}$, since $S_i$ is dominant, there exists $X_i\subseteq V(G)$ with $\alpha(G[X_i])\leq c$ such that $N_G[S_i]\cup X_i$ is a $\we$-balanced separator in $G$. Let $$Z=X_1\cup X_2\cup Y.$$
    Then $\alpha(G[Z])\leq 2c+d$, and therefore $Z$ is not a $\we$-balanced separator in $G$. It follows that there is a component $C$ of $G\setminus Z$ with $\we(V(C))>1/2$. For each $i\in \{1,2\}$, since $N_G[S_i]\cup X_i$ is a $\we$-balanced separator in $G$, and since $C$ is a connected induced subgraph of $G$ with $\we(V(C))>1/2$ and $V(C)\cap X_i\subseteq V(C)\cap Z=\varnothing$, it follows that $V(C)\cap N_G[S_i]\neq \varnothing$. Specifically, since $C$ is connected and $S_1$ and $S_2$ are anticomplete in $G$, it follows that there exists $P\in \mca{P}_G(S_1,S_2)$ with $P^*\subseteq V(C)$. But now since $(S_1\cup S_2)\cap Y=\varnothing$ and $P^*\cap Y\subseteq V(C)\cap Y\subseteq V(C)\cap Z=\varnothing$, we have $V(P)\cap Y=\varnothing$, contradicting the fact that $Y$ is a hitting set for $\mca{P}_G(S_1,S_2)$. This completes the proof of \Cref{lem:ta-vs-sep-domin}.
\end{proof}

We also need the following (more precisely, a slight improvement, given by \Cref{cor:sep-domin}):

\begin{lemma}[Chudnovsky, Codsi, Lokshtanov, Milani\v{c}, Sivashankar; 1.3 in \cite{ti5}]\label{lem:sep-domin}
    For every graph $F$ isomorphic to a subdivision of $K_{1,3}$ and every wall $W$, there is a constant $b_{\ref{lem:sep-domin}}=b_{\ref{lem:sep-domin}}(F,W)$ such that every $F$-free graph with no induced subgraph isomorphic to the line graph of any subdivision of $W$ is $b_{\ref{lem:sep-domin}}$-separator-dominated.
\end{lemma}

\begin{corollary}\label{cor:sep-domin}
      For every subdivided multiclaw $F$ and every wall $W$, there is a constant $b_{\ref{cor:sep-domin}}=b_{\ref{cor:sep-domin}}(F,W)$ such that every $F$-free graph with no induced subgraph isomorphic to the line graph of any subdivision of $W$ is $b_{\ref{cor:sep-domin}}$-separator-dominated.
\end{corollary}
\begin{proof}
    Let $r=|V(F)|$. Let $F^+$ be the $(r-1)$-subdivision of $K_{1,3}$. Let
    $$b_{\ref{cor:sep-domin}}=b_{\ref{cor:sep-domin}}(F,W)=3r^2+b_{\ref{lem:sep-domin}}(F^+,W).$$

Let $\we: V(G)\to [0,1]$ be a normal function. Our goal is to show that there exists $S\subseteq V(G)$ with $|S|\leq b_{\ref{cor:sep-domin}}$ such that $N_G[S]$ is a $\we$-balanced separator in $G$. Let $\mca{F}$ be a maximal set of pairwise anticomplete induced subgraphs of $G$, each isomorphic to $F^+$. Since $G$ is $F$-free, it follows that $|\mca{F}|\leq r-1$. Let $U=\bigcup_{H\in \mca{F}}V(H)$. Then $|U|\leq (r-1)(3r+1)<3r^2$, and $G_0=G\setminus N_G[U]$ is $F^+$-free since $\mca{F}$ is maximal.

    Let $\we_0:V(G_0)\to [0,1]$ be the normal function with the rule $\we_0(v)=\we(v)/\we(V(G_0))$. Since $G_0$ is $F^+$-free with no induced subgraph isomorphic to the line graph of any subdivision of $W$, it follows from \Cref{lem:sep-domin} that $G_0$ is $b_{\ref{lem:sep-domin}}(F^+,W)$-separator-dominated; that is, there exists $S_0\subseteq V(G_0)$ with $|S_0|\leq b_{\ref{lem:sep-domin}}(F^+,W)$ such that $\we_0(V(C))\leq 1/2$ for every component $C$ of $G_0\setminus N_{G_0}[S_0]$. Since $G\setminus (N_{G}[U]\cup N_{G}[S_0])=G_0\setminus N_{G_0}[S_0]$, it follows that $\we(V(C))=\we_0(V(C))\we(V(G_0))\leq \we_0(V(C))\leq 1/2$ for every component $C$ of $G\setminus (N_{G}[U]\cup N_{G}[S_0])$; consequently, $N_{G}[U]\cup N_{G}[S_0]=N_{G}[U\cup S_0]$ is a $\we$-balanced separator in $G$. Moreover, we have $|U\cup S_0|\leq |U|+|S_0|\leq 3r^2+b_{\ref{lem:sep-domin}}(F^+,W)=b_{\ref{cor:sep-domin}}(F,W)$. This completes the proof of \Cref{cor:sep-domin}.
\end{proof}

Finally, we are ready to prove \Cref{thm:ta-vs-sep}, which we restate:

\tavssep*
\begin{proof}
    Let
$b_{\ref{thm:ta-vs-sep}}=b_{\ref{cor:sep-domin}}(F,W)$. Let $G$ be an $F$-free graph with no induced subgraph isomorphic to the line graph of any subdivision of $W$. By \Cref{cor:sep-domin}, every induced subgraph of $G$ is $b_{\ref{thm:ta-vs-sep}}$-separator-dominated. Hence, since $G$ is $(b_{\ref{thm:ta-vs-sep}},c)$-loosely $d$-$\alpha$-separable, by \Cref{lem:ta-vs-sep-domin}, we have $\ta(G)\leq 10c+5d$, as desired.
\end{proof}

\bibliographystyle{abbrv}
\bibliography{ref}  

\end{document}